\documentclass{article}
\usepackage{amsmath,amsthm,amssymb,amscd}

\newcommand{\ga}{\alpha}
\newcommand{\gb}{\beta}
\renewcommand{\gg}{\gamma}

\newcommand{\gw}{\omega}

\newcommand{\gS}{\Sigma}
\newcommand{\gs}{\sigma}
\newcommand{\eps}{\varepsilon}

\newcommand{\supp}{\mathrm{supp}}

\newcommand{\dom}{\mathrm{dom}}

\newcommand{\rng}{\mathrm{rng}}

\newcommand{\coll}{\mathrm{Coll}}

\newcommand{\hefi}{\mathrm{HF}}

\newtheorem{theorem}{Theorem}[section]

\newtheorem{claim}[theorem]{Claim}
\newtheorem{corollary}[theorem]{Corollary}

\newtheorem{proposition}[theorem]{Proposition}

\theoremstyle{definition}
\newtheorem{definition}[theorem]{Definition}
\newtheorem{example}[theorem]{Example}
\newtheorem{question}[theorem]{Question}

\title{Triangles and Vitali sets\footnote{2020 AMS subject classification 03E35, 14P99, 05C15.}}

\author{
Jind{\v r}ich Zapletal\\
University of Florida\\
zapletal@ufl.edu}

\begin{document}
\maketitle

\begin{abstract}
It is consistent relative to an inaccessible cardinal that ZF+DC holds, the hypergraph of equilateral triangles in Euclidean plane has countable chromatic number, while there is no Vitali set.
\end{abstract}

\section{Introduction}

Classification of analytic graphs and hypergraphs from various points of view is a subject matter frequently discussed in descriptive set theory and descriptive combinatorics. Geometric set theory adds its own angle: independence results in ZF+DC choiceless set theory regarding the existence of various structures (e.g., colorings or perfect matchings) related to analytic (hyper)graphs. One can consider for example the following broad question:

\begin{question}
For which analytic hypergraphs $\Gamma$ on Polish spaces is it consistent with ZF+DC that $\Gamma$ has countable chromatic number and no Vitali set exists?
\end{question}

\noindent Note that existence of Vitali set is equivalent in ZF+DC to the countable chromatic number of the Vitali equivalence relation on $\mathbb{R}$. \cite[Chapter 11]{z:algebraic} contains various results in this direction for locally countable graphs and hypergraphs. Difficulties quickly mount when one steps out of the locally countable realm. As one interesting and traditional class of cases, the algebraic graphs and hypergraphs on Euclidean spaces are combinatorially quite complex, but none of them seem to have any relationship with the Vitali equivalence relation. In an encouraging development, the question has been resolved in the affirmative for all algebraic graphs $\Gamma$ containing no perfect clique \cite{z:ngraphs}. In this paper, I resolve a quite special case in arity three, which requires several additional ideas.

\begin{definition}
Let $\langle X, \cdot\rangle$ be a standard Borel abelian group. A hypergraph $\Gamma$ of arity three on $X$ is a \emph{slim linear hypergraph} if there are Borel homomorphisms $g_0, g_1, g_2\colon X\to X$ such that

\begin{enumerate}
\item all of $g_0, g_1, g_2$ and $g_0+g_1$, $g_1+g_2$, and $g_0+g_2$ are injections;
\item $\Gamma$ is the set of all triples $\{x_0, x_1, x_2\}$ of pairwise distinct points such that under suitable enumeration, $g_0(x_0)+g_1(x_1)+g_2(x_2)=0$.
\end{enumerate}
\end{definition}

\noindent In most interesting cases, the sum $g_0+g_1+g_2$ is equal to the zero homomorphism, and then the hypergraph $\Gamma$ will be invariant under translations. Slim linear hypergraphs carry geometric meaning:

\begin{example}
Let $c\in\mathbb{C}$ be a non-real complex number, and consider the hypergraph on $\mathbb{C}$ given by the arithmetic identity $c\cdot (x_1-x_0)=x_2-x_0$. It is the set of all triangles in the complex plane directly similar to the triangle with vertices $0, 1, c$.
\end{example}

\begin{theorem}
\label{maintheorem}
Let $\langle X, +\rangle$ be a standard Borel abelian group, and let $\Gamma$ be a finite union of slim linear hypergraphs of arity three on $X$. Then in a balanced forcing extension of the choiceless Solovay model, the following items hold:

\begin{enumerate}
\item ZF+DC holds and $\Gamma$ has countable chromatic number;
\item there is no Vitali set;
\item there is no complete countable section for the $\mathbb{E}_1$ equivalence relation.
\end{enumerate}
\end{theorem}

\begin{corollary}
Relative to an inaccessible cardinal, it is consistent with ZF+DC that the hypergraph of equilateral triangles in $\mathbb{R}^2$ has countable chromatic number, yet there is no Vitali set.
\end{corollary}

\noindent The theorem understates the understanding of the resulting model. In particular, I prove that a certain graph simpler than the Vitali equivalence relation--the diagonal Hamming graph--has uncountable chromatic number in the resulting model. 

In a number of natural cases, I can show that it is not entirely easy to color slim linear hypergraphs. No such coloring will appear in a compactly balanced forcing extension of the choiceless Solovay model as per the following theorem. The broad class of compactly balanced forcings was introduced in \cite[Section 9.3]{z:geometric}; it includes in particular the poset for adding a Ramsey ultrafilter, or the poset for adding a linear ordering of a quotient space of any Borel equivalence relation.

\begin{theorem}
\label{maintheorem2}
Let $\langle X, +\rangle$ be a Polish abelian group, let $n\geq 3$ be a number, and let $g_i\colon X\to X$ for $i\in n$ be continuous homomorphisms such that

\begin{enumerate}
\item $\gS_{i\in n}g_i=0$;
\item every open neighborhood of $0$ in $X$ contains pairwise distinct nonzero points $x_i$ for $i\in n$ such that $\gS_{i\in n}g_i(x_i)=0$.
\end{enumerate}

\noindent Then, in every compactly balanced extension of the choiceless Solovay model, the hypergraph $\Gamma$ on $X$ consisting of all tuples $\{x_i\colon i\in n\}$ of pairwise distinct elements of $X$ such that $\gS_{i\in n}g_i(x_i)=0$ has uncountable chromatic number.
\end{theorem}

\begin{corollary}
Relative to an inaccessible cardinal, it is consistent with ZF+DC that there is a non-principal ultrafilter on $\gw$, yet the hypergraph of equilateral triangles in $\mathbb{R}^2$ has uncountable chromatic number.
\end{corollary}

There are many open questions left. Since the equilateral triangle hypergraph in $\mathbb{R}^3$ is not linear, methods of this paper do not apply to it, and it is the simplest algebraic hypergraph for which I do not know if its countable chromatic number implies existence of a Vitali set in ZF+DC. I do not know whether the chromatic number of the Kechris--Solecki--Todorcevic graph $\mathbb{G}_0$ \cite{kechris:chromatic} is countable or uncountable in the resulting model. Finally, one should be able to remove the inaccessible cardinal assumption from the proof, but I have not done the footwork necessary for that.

The architecture of the paper is probably somewhat surprising to a reader who is not familiar with the mysterious ways of geometric set theory \cite{z:geometric}. In Section~\ref{remaindersection}, I define the central combinatorial tool of the paper: the remainder graphs of a ternary hypergraph. This is an interesting object regardless of any forcing considerations, and the main point of this paper is that in the case of  linear hypergraphs, the remainder graphs can be handled efficiently. In Section~\ref{cccsection}, I define a couple of c.c.c.\ finite condition coloring posets for various graphs and hypergraphs. A fine evaluation of their chain condition situation is necessary to exclude the Vitali set in the model for Theorem~\ref{maintheorem}. This part of the paper is reminiscent of certain computations of Todorcevic \cite{todorcevic:examples} regarding similar finite condition c.c.c.\ posets. In Section~\ref{balancedsection}, I define a $\gs$-closed balanced poset $P$ adding a coloring of the hypergraph $\Gamma$ of Theorem~\ref{maintheorem}.  The model for Theorem~\ref{maintheorem} is then a $P$-extension of a choiceless Solovay model derived from an inaccessible cardinal. Section~\ref{finalsection} then checks the properties of this model, using the remainder graphs and c.c.c.\ coloring posets identified earlier. Finally, Section~\ref{problematicsection} shows that there are natural algebraic hypergraphs for which in ZF countable chromatic number has consequences contradicting the conclusion of Theorem~\ref{maintheorem}.

The terminology of the paper follows the set theoretic standard of \cite{jech:newset}, in matters of descriptive set theory that of \cite{kechris:classical}, in matters of geometric set theory that of \cite{z:geometric}. $\mathbb{E}_1$ is the equivalence relation on the set $\mathbb{R}^\gw$ connecting points $x, y$ if the set $\{n\in\gw\colon x(n)\neq y(n)\}$ is finite. The \emph{diagonal Hamming graph} $\mathbb{H}_{<\gw}$ is the graph on the set of all functions $x\in\gw^\gw$ such that $\forall n\ x(n)\leq n$, connecting points $x, y$ if there is exactly one $n\in\gw$ such that $x(n)\neq y(n)$. $\hefi$ denotes the set of all hereditarily finite sets.

\section{Remainder graphs}
\label{remaindersection}

The consistency result depends on several combinatorial observations. The first seems critical and irreplaceable in the argument, and its limitations seem to be the main obstacle to natural generalizations of Theorem~\ref{maintheorem}. Consider the following definition.

\begin{definition}
Let $\Gamma$ be a hypergraph on a Polish space $X$, and let $A\subset X$ be a set. The \emph{remainder graph} is the graph $\Gamma_A$ on $X\setminus A$ connecting distinct points $x, y$ if there is $z\in A$ such that $\{x, y, z\}\in\Gamma$ holds.
\end{definition}

\noindent The most natural case of remainder graphs appears when $A$ is a suitably algebraically closed subset of $X$, possibly uncountable. I will be interested in the chromatic number of the remainder graphs. This looks like a completely intractable issue, but some cases do allow for informative insight. Consider the following examples.

\begin{example}
Let $\Gamma$ be the hypergraph on $\mathbb{R}$ of all triples $\{x, y, z\}$ of pairwise distinct elements which (in some ordering) satisfy the identity $x-2y+z=0$. Let $A\subset\mathbb{R}$ be a subfield. Then the remainder graph contains no circuits of odd length, and it contains an $|A|\times|A|$-biclique.
\end{example}

\begin{proof}
First, I will show that there are no circuits of odd length in the remainder graph. For this, consider the group $G$ of all non-constant linear functions from $\mathbb{R}$ to itself with the composition operation, acting on $\mathbb{R}$ by application. For each such a function $g(u)=au+b$, call $a$ the \emph{slope} of the function. The slope function is obviously a homomorphism from the group $G$ to the multiplicative group of nonzero real numbers. Each element of $G$ except for the identity has at most one fixed point.

Now, let $n\in\gw$ be a number and $\langle x_i\colon i\leq n\rangle$ be a path in the remainder graph with $x_n=x_0$; I must show that $n$ is even. To this end, for each $i\in n$ find $z_i\in V$ such that $\{x_i, x_{i+1}, z_i\}\in\Gamma$. There are three possible mutually exclusive cases how this can occur:

\begin{itemize}
\item $x_i-2x_{i+1}+z_i=0$. In this case, let $g_i\colon X\to X$ be the function defined by $g_i(u)=(u+z_i)/2$;
\item $x_i-2z_i+x_{i+1}=0$. In this case, let $g_i\colon X\to X$ be the function defined by $g_i(u)=2z_i-u$;
\item $z_i-2x_i+x_{i+1}=0$. Here, let $g_i(u)=2u-z_i$.
\end{itemize} 

\noindent All the functions $g_i$ for $i\in n$ are coded in $F_0$, and $g_i(x_i)=x_{i+1}$. The composition of all the elements $g_i\in G$ for $i\in n$ is a linear function coded in $A$, and it has $x_0\notin A$ as a fixed point, therefore it has to be the identity. Now consider the slopes of the functions $g_i$: in the first case, it is $1/2$, in the second case, it is $-1$, in the third case, it is $2$. The product has to give $1$, the slope of the identity. Thus, the first and third case must be used equal number of times, and the second case must be used even number of times. In conclusion, $n$ is even as desired.

Now, to produce the required biclique in the remainder graph, let $B_0, B_1\subset A$ be disjoint sets of reals of cardinality $|A|$. Let $\eps>0$ be a real number in $\mathbb{R}\setminus A$. Consider the set $C_0=B_0+\eps$ and the set $C_1=(B_1+\eps)/2$. These are disjoint subsets of $\mathbb{R}\setminus A$, and they form an $|A|\times|A|$ bi-clique in the remainder graph. To see this, let $x_0\in B_0$ and $x_1\in B_1$ be points. Then $(x_0+\eps)-(2(x_1+\eps)/2)=x_0-x_1\in A$, so $\{x_0+\eps, (x_1+\eps)/2, x_0-x_1\}\in\Gamma$ and $x_0+\eps, (x_1+\eps)/2$ are connected in the remainder graph as desired.
\end{proof}

\begin{example}
Let $\Gamma$ be the hypergraph of equilateral triangles in $\mathbb{R}^3$. Let $F$ be any real closed subfield of $\mathbb{R}$ properly smaller than $\mathbb{R}$, and let $A=F^3$. Then in the remainder graph, one can find a triangle. Just let $\eps>0$ be a real number $\mathbb{R}\setminus F$, let $z\in A$ be an arbitrary point, and find any regular tetrahedron with one vertex equal to $z$ and length of all sides equal to $\eps$. The other three vertices cannot belong to $A$ since together with $z$ each of them computes $\eps$, and they form a triangle in the remainder graph.
\end{example}

\begin{example}
Let $\Gamma$ be the hypergraph of isosceles triangles on $\mathbb{R}^2$. Let $F$ be any real closed subfield of $\mathbb{R}$ properly smaller than $\mathbb{R}$, and let $A=F^3$. Then the remainder graph contains a perfect clique. To see this, let $\eps>0$ be a real number which is not in $F$, and let $z\in A$ be an arbitrary point. Consider the set $C$ of all points in $\mathbb{R}^2$ which are at a distance $\eps$ from $z$. It is immediate to see that $C$ contains no points from $A$ (since every point in $C$ together with $z$ computes $\eps$) and $C$ is a clique in the remainder graph.
\end{example}

\noindent Finally, we get to the hypergraph of immediate interest to the present paper.

\begin{definition}
Let $\langle X, +\rangle$ be a standard Borel abelian group. Let $\Gamma=\bigcup_{i\in n}\Gamma_i$ be a finite or countable union of slim linear hypergraphs of arity three, for $n\leq\gw$. Let $\Gamma_i$ be generated by continuous homomorphisms $g_j^i$ for $j\in 3$. A set $A\subset X$ is $\Gamma$-closed if it is a subgroup of $X$, it is closed under each homomorphism $g_j^i$ and its inverse for $i\in n$ and $j\in 3$, and also closed under the inverses of the homomorphisms $g_j^i+g_k^i$ for $i\in\gw$ and $j, k$ distinct indices in $3$.
\end{definition}

\noindent Note that all the homomorphisms mentioned in the above definitions are injective per the definition of a slim linear hypergraphs, so their inverses are well-defined partial functions. The definition of $\Gamma$-closed set seems to depend on the choice of the decomposition of $\Gamma$ into slim hypergraphs and their generating homomorphisms; I will keep such decomposition and its generating homomorphisms always fixed. In this way, every infinite subset of $X$ has a smallest $\Gamma$-closed superset which is of the same cardinality, and an increasing union of $\Gamma$-closed sets is $\Gamma$-closed.

\begin{example}
\label{criticalexample}
Let $\langle X, +\rangle$ be a standard Borel abelian group. Let $\Gamma$ be a finite (resp. countable) union of slim linear hypergraphs of arity three. Let $A\subset X$ be a $\Gamma$-closed set. Then there is a homomorphism of the remainder graph $\Gamma_A$ to a locally finite (resp. locally countable) graph.
\end{example}

\begin{proof}
et $\Gamma=\bigcup_{i\in n}\Gamma_i$ be a finite or countable union of slim linear hypergraphs of arity three, for $n\leq\gw$. Let $\Gamma_i$ be generated by continuous homomorphisms $g_j^i$ for $j\in 3$. Let $E$ be the equivalence relation on $X\setminus A$ connecting distinct points $x_0, x_1$ if $x_0-x_1\in A$. Let $Y$ be the set of all $E$-classes, and define the graph $\Theta$ on $Y$ which connects classes $c\neq d$ if there are representatives $x\in c$ and $y\in d$ which are connected in the remainder graph. I will show that that the graph $\Theta$ is locally finite (resp. locally countable), and the map $f\colon X\to Y$, $f(x)=[x]_E$ is a homomorphism of the remainder graph to $\Theta$.

To show that the map $f$ is a homomorphism, it is only necessary to show that if distinct points $x_0, x_1\in X\setminus A$ are connected in the remainder graph, then they are not $E$-equivalent. Suppose towards a contradiction that they are. Let $i\in n$ and $x_2\in A$ be such that $g_0^i(x_0)+g_1^i(x_1)+g_2^i(x_2)=0$, and let $y\in A$ be such that $x_0+y=x_1$.
Since $A$ is a subgroup of $X$ closed under the homomorphisms, plugging in for $x_1$ it becomes clear that $(g_0+g_1)(x_0)\in A$. Since $A$ is closed under the inverse of $g_0+g_1$, it follows that $x_0\in A$, contradicting the initial choice of $x_0$. 

To show that the graph $\Theta$ is locally finite (resp. locally countable), it will be enough to show that if $c, d, d'$ are $E$-classes, $i\in n$ is an index, and there are points $x_0, x'_0\in c$, $x_1\in d_0$, $x'_1\in d'$, and $x_2, x'_2\in A$ such that $g^i_0(x_0)+g_1^i(x_1)+g_2^i(x_2)=0$ and $g^i_0(x'_0)+g_1^i(x'_1)+g_2^i(x'_2)=0$ both hold, then $d=d'$. To see that, subtract the latter equality from the former and get $g^i_0(x_0-x'_0)+g^i_1(x_1-x'_1)\in A$. The first summand here belongs to $A$ as $x_0\mathrel Ex'_0$ holds. As $A$ is closed under the inverse of $g^i_1$, it follows that $x_1-x'_1\in A$ holds. Thus $d=d'$ as desired.
\end{proof}

\begin{corollary}
\label{chromaticcorollary1}
Let $\langle X, +\rangle$ be a standard Borel abelian group. Let $\Gamma$ be a union of countably many slim linear hypergraphs of arity three on $X$. Whenever $A\subset X$ is a $\Gamma$-closed set, the chromatic number of $\Gamma(A)$ is countable.
\end{corollary}

\noindent One can use the chromatic number of remainder graphs to prove that the chromatic number of $\Gamma$ itself is countable. In addition, the mechanics of this proof will drive the consistency arguments in Section~\ref{finalsection}. For a notational convenience, my (hyper)graph colorings of hypergraphs will often have the countable set $\hefi$ of hereditarily finite sets as a co-domain.

\begin{definition}
\label{coherentdefinition}
Let $\langle X, +\rangle$ be a Polish abelian group. Let $\Gamma$ be a countable union of slim linear hypergraphs.
A \emph{coherent sequence of colorings} is a tuple $\langle I, \leq, A_i, c_i, d_i\rangle$ where 

\begin{enumerate}
\item $\langle I, \leq\rangle$ be a linearly ordered set;
\item $\langle A_i\colon i\in I\rangle$ is an inclusion-increasing sequence of $\Gamma$-closed subsets of $X$;
\item for every $x\in\bigcup_iA_i$, there is a smallest index $i\in I$ such that $x\in A_i$;
\item $c_i\colon A_i\to \hefi$ is a $\Gamma$-coloring;
\item if $i$ is the smallest element of $I$ then $d_i=0$; otherwise, write $\Delta_i$ for the remainder graph $\Delta(\bigcup_{j<i}A_j)$, and $d_i\colon A_i\setminus\bigcup_{j<i}A_j\to\hefi$ is a $\Gamma_{\bigcup_{j<i}A_j}$-coloring. 
\end{enumerate}
\end{definition}

\noindent Note that the linearly ordered set $I$ must have a smallest element, namely the smallest $i$ such that $1\in F_i$. One way to satisfy item (3) is to insist that $\leq$ is a well-ordering; however, I will have an opportunity to use an ill-founded linear order as well.

\begin{definition}
\label{amalgamationdefinition}
The \emph{amalgamation} of a coherent sequence $\langle I, \leq, A_i, c_i, d_i\rangle$ of colorings is the function $e$ on $\bigcup_i A_i$ defined as follows. For every $x\in\bigcup_iA_i$ find the smallest $i\in I$ such that $x\in A_i$. If $i$ is the smallest element of $I$, then let $e(x)=f_i(x)$; otherwise, let $e(x)=\langle f_i(x), g_i(x)\rangle$.
\end{definition}

\noindent Note that the amalgamation extends $f_i$ for the smallest element $i\in I$.

\begin{proposition}
\label{amalgamationproposition}
The amalgamation of a coherent sequence of colorings is a $\Gamma$-coloring.
\end{proposition}

\begin{proof}
Let $\langle I, \leq, A_i, c_i, d_i\rangle$ be a coherent sequence of colorings. Write $Y=\bigcup_iA_i$. For each $x\in X$ write $i(x)$ for the least $i\in I$ such that $x\in A$. Let $e$ be the amalgamation of the sequence. Let $a\subset X$ be a $\Gamma$-hyperedge and work to show that $e\restriction a$ is not constant. Let $i$ be the $\leq$-largest element of the finite set $\{i(x)\colon x\in a\}$. There are three configurations to consider; their discussion completes the proof.

\noindent\textbf{Case 1.} There is exactly one point $x\in a$ such that $i(x)=i$. This is in fact impossible, since the subgroup $\bigcup_{j<i}A_j$ is $\Gamma$-closed.

\noindent\textbf{Case 2.} There are exactly two points $x\in a$ such that $i(x)=i$. In this case, these two points are $\Delta_i$-connected and they receive distinct $d_i$-colors, as $d_i$ is a $\Delta_i$-coloring. In consequence, they also receive distinct $e$-colors.

\noindent\textbf{Case 3.} All three points $x\in a$ have $i(x)=i$. In this case, $a$ is not $c_i$-monochromatic, since $c_i$ is a $\Gamma$-coloring. As a result, $a$ is not $e$-monochromatic either.
\end{proof}

\begin{corollary}
\label{chromaticcorollary2}
Let $\langle X, +\rangle$ be a Polish abelian group. Let $\Gamma$ be a countable union of countably many slim linear hypergraphs of arity three on $X$. The chromatic number of $\Gamma$ is countable.
\end{corollary}

\begin{proof}
By induction on the cardinality of a $\Gamma$-closed set $A\subset X$ I will argue that there is a $\Gamma$-coloring from $A$ to $\hefi$. This is clear if $|A|=\aleph_0$, since one can choose any injection as the coloring. Suppose now that $\kappa$ is an uncountable cardinal, $|A|=\kappa$, and the statement has been proved for all $\Gamma$-closed sets of cardinality smaller than $\kappa$. Express $F=\bigcup_{\gb\in\ga}F_\gb$ as an increasing union of $\Gamma$-closed sets of cardinality smaller than $\kappa$. Use the induction hypothesis to find a $\Gamma$-coloring $c_\gb\colon A_\gb\to\hefi$ and Example~\ref{criticalexample} to find a $\Delta(\bigcup_{\gamma\in\gb}A_\gamma)$-coloring $d_\gb\colon A_\gb\setminus\bigcup_{\gamma\in\gb}A_\gamma\to\hefi$, this for every $\gb\in\ga$. Let $c\colon A\to\hefi$ be the amalgamation of $\langle c_\gb, d_\gb\colon\gb\in\ga\rangle$. Proposition~\ref{amalgamationproposition} confirms that this is a $\Gamma$-coloring, completing the induction step and the proof.
\end{proof}

\section{C.c.c.\ coloring posets}
\label{cccsection}

Another issue seemingly unrelated to the consistency result in question is the possibility of forcing colorings of given (hyper)graphs by finite conditions. If $\Gamma$ is a hypergraph on a set $X$ of countable chromatic number, there may not be a coloring definable from $\Gamma$--this is the basic difficulty of descriptive combinatorics. However, there still may be a definable poset adding a coloring which has useful regularity properties. The poset I use in this paper is the simplest one:

\begin{definition}
Let $\Gamma$ be a hypergraph on a set $X$. The \emph{finite approximation coloring poset} $R(\Gamma)$ is the poset of all partial finite colorings $p\colon X\to\gw$, ordered by reverse inclusion.
\end{definition}

\noindent The countable set used as the target of the colorings is of course immaterial; I may use the set $\hefi$ instead of $\gw$. The property of the coloring posets critical for this paper is the following.

\begin{definition}
Let $R$ be a poset.

\begin{enumerate}
\item A set $A\subset R$ is \emph{Ramsey-centered} if for every $n\in\gw$ there is $m\in\gw$ such that for every $m$-tuple $\langle r_i\colon i\in m\rangle$ of elements of $A$ there is a set $b\subset m$ of cardinality $n$ such that the set $\{r_i\colon i\in b\}$ has a common lower bound in $R$.
\item if $z$ is a set, then $R$ is $z$-\emph{definably $\gs$-Ramsey-centered} if $R$ can be covered by an $\gw$-sequence of Ramsey-centered sets which is defined from $z$
\end{enumerate}
\end{definition}

\noindent There are several situations in which I can show that the finite approximation coloring poset is suitably definably Ramsey-centered. Some are exhibited in \cite[Section 11.6]{z:geometric}. The cases I need in this paper are recorded in the following two similar theorems with similar proofs.

\begin{theorem}
\label{1theorem}
Let $X$ be a set. Let $\Gamma$ be a hypergraph of arity three on $X$ such that there is a number $d\in\gw$ such that for every $x_0, x_1\in X$ the set $\{x_2\in X\colon \{x_0, x_1, x_2\}\in\Gamma\}$ has cardinality at most $d$. Then the finite approximation coloring poset for $\Gamma$ is $X, \Gamma$-definably $\gs$-Ramsey centered.
\end{theorem}

\begin{proof}
Let $A_{kl}\subset R(\Gamma)$ be the set of all conditions whose domain has cardinality $k\in\gw$ and the range is a subset of $l\in\gw$. It will be enough to show that $A_{kl}$ is Ramsey-centered.

To this end, let $n\in\gw$ be a number. Increasing $n$ if necessary I may assume that $n>d\cdot [2k]^2$ and $n>3l$. Let $m$ be a number such that $m\to (n)^3_{9k+1}$. It will be enough to show that among any $m$ many elements of $A_i$ one can find $n$ many with a common lower bound. To this end, let $\{p_i\colon i\in m\}\subset A_k$ be a set. For each $i\in m$, fix an enumeration $\{x_j^i\colon j\in k\}$ of its domain. Consider the following map $f$ on $[m]^3$. If $u\in [m]^3$ is a set, let $i_0<i_1<i_2$ be its elements in increasing order and define $f(u)$ by a split into cases:

\begin{itemize}
\item if there is an index $b\in 3$ and an index $j\in k$ such that $x_j^{i_b}$ belongs to at least one of $\dom(p_{i_c})$ for $c\neq b$ and it is not equal to at least one $x_j^{i_c}$ for $c\neq b$, then the functional value will be $f(u)=\langle 0, b, j\rangle$ for some such $b$ and $j$;
\item if the previous item fails and there are $b\in 3$ and $j\in k$ such that for some $c\in 3$ distinct from $b$, $x_j^{i_b}=x_j^{i_c}$ and $p_{i_b}(x_j^{i_b})\neq p_{i_c}(x_j^{i_c})$, then $f(u)=\langle 1, b, j\rangle$ for some such $b, j$;
\item if the previous items fail and there are $b\in 3$ and $j\in k$ such that there is a hyperedge $e\in\Gamma$ which is contained in $\bigcup_{i\in u}\dom(p_i)$ and such that $x_j^{i_b}$ is the only element of $e$ which is not in $\dom(p_{i_c})$ for any $c\neq b$, then let $f(u)=\langle 2, b, j\rangle$ for some such $b, j$;
\item if all of the previous items fail then $f(u)=$OK.
\end{itemize}

\noindent By the Ramsey assumption on the number $m$, there is a set $a\subset m$ of cardinality $n$ which is homogeneous for the coloring $f$. I will argue that the homogeneous color is OK, the sets $\dom(p_i)$ for $i\in a$ form a $\Delta$-system, $\bigcup_{i\in a}p_i$ is a function, and every $\Gamma$-edge in $\dom(\bigcup_{i\in a}p_i)$ is a subset of $\dom(p_i)$ for some $i\in a$. That edge will not be monochromatic since $p_i$ is a $\Gamma$-coloring. In conclusion, $\bigcup_{i\in a}p_i$ is a common lower bound in $R(\Gamma)$ of the conditions $p_i$ for $i\in a$. This will conclude the proof of the theorem. I proceed in a sequence of claims.

\begin{claim}
The homogeneous color is not $\langle 0, b, j\rangle$ for any $b\in3$ and $j\in k$.
\end{claim}

\begin{proof}
Suppose towards a contradiction that this occurs; for definiteness, assume that $b=2$. Then there are only $2k$ many options for $x_j^i$ for $i\in a$ (they all must come from the first two elements of the set $a$), so they must repeat, and there must be a set $u\in [a]^3$ such that the value of $x_j^i$ for $i\in u$ does not depend on $i$. This contradicts the assumption that $f(u)=\langle 0, b, j\rangle$.
\end{proof}

\noindent It follows that writing $G=\{j\in k\colon$ the value of $x_j^i$ does not depend on $i\in a\}$, the sets $\dom(p_i)$ for $i\in a$ form a $\Delta$-system with heart $H=\{x_j^i\colon i\in a, j\in G\}$. To see this, note that if $i_0\neq i_1$ were two distinct elements of $a$ and $j\in k$ were such that $x_j^{i_0}\in \dom(p_{i_1})\setminus\{x_j^{i_1}\}$, then this would trigger the first item above for any $u\in [a]^3$ containing both elements $i_0, i_1$, contradicting the claim.

\begin{claim}
The homogeneous color is not $\langle 1, b, j\rangle$ for any $b\in 3$ and $j\in k$.
\end{claim}

\begin{proof}
If it were, $j\in G$ would have to occur. There are only $l$ many colors available for $p_i(x_j^i)$ for $i\in a$, which means that one of them has to repeat at least three times. Let $u\in [a]^3$ be such that $p_i(x_j^i)$ for $i\in u$ does not depend on $i$. Then $f(u)=\langle 1, b, j\rangle$ must fail, a contradiction.
\end{proof}

\noindent It follows that $\bigcup_{i\in a}p_i$ is a function. Any disagreement of distinct conditions on the common part $H$ of their domain would trigger the second item above for some set $u\in [a]^3$, contradicting the claim.

\begin{claim}
The homogeneous color is not $\langle 2, b, j\rangle$ for any $b\in 3$ and $j\in k$.
\end{claim}

\begin{proof}
Suppose towards a contradiction that this occurs; for definiteness, assume that $b=2$. Then there are at most $d\cdot [2k]^2$ many options for the value of $x_j^i$ as $i$ varies over all elements of $a$ by the assumption on the hypergraph $\Gamma$. This means that these values have to repeat, and there will be a set $u\in [a]^3$ such that $x_j^i$ is the same for all $i\in u$. Such $u$ cannot have $f(u)=\langle 2, b, j\rangle$.
\end{proof}

\noindent Now, let $e$ be a $\Gamma$-hyperedge in $\dom(\bigcup_{i\in a}p_i)$; I need to argue that there is a single $i\in a$ such that $e\subset\dom(p_i)$. Otherwise, there would be $i\in a$ such that $e\cap\dom(p_i)\setminus H$ is a singleton, which would trigger the third item for any set $u\in [a]^3$ containing $i$ and the other two indices needed to reconstruct $e$, contradicting the claim. This completes the proof of the theorem. 
\end{proof}

\begin{theorem}
\label{2theorem}
Let $X$ be a set and let $\Gamma$ be a graph such that for every $x_0\in X$, the set $\{x_1\in X\colon \{x_0, x_1\}\in\Gamma\}$ is finite. Then $R(\Gamma)$ is $X, \Gamma$-definably Ramsey-$\gs$-centered.
\end{theorem}

\begin{proof}
For each $k, l, d\in\gw$ let $A_{kld}=\{p\in R(\Gamma)\colon |\dom(p)|=k, \rng(p)\subset l$, and each point in $\dom(p_i)$ has at most $d$-many neighbors$\}$. It will be enough to show that each set $A_{kld}$ is Ramsey-centered.

To this end, let $k, l, d, n\in\gw$ be arbitrary. Increasing the number $n$ if necessary, I may assume that $n>d+2, k+2$. Let $m\in\gw$ be a number such that $m\to (n)^2_{6k+1}$. It will be enough to show that among any $m$ many elements of $A_{kld}$ one can find $n$ many with a common lower bound, i.e. such that their union is still a partial $\Gamma$-coloring.

To do that, let $\{p_i\colon i\in m\}\subset A_{kld}$ be any set. For each $i\in m$, fix an enumeration $\langle x_j^i\colon i\in k\rangle$ of $\dom(p_i)$. Consider the following map $f$ on $[m]^2$. If $u\in [m]^2$ is a set, let $i_0<i_1$ be its elements in increasing order and define $f(u)$ by a split into cases:

\begin{itemize}
\item if there is an index $b\in 2$ and an index $j\in k$ such that $x_j^{i_b}\neq x_j^{i_{1-b}}$ yet $x_j^{i_b}\in v_{i_{1-b}}$, then the functional value will be $f(u)=\langle 0, b, j\rangle$ for some such $b$ and $j$;
\item if the previous item fails and there is $j\in k$ such that $x_j^i$ does not depend on $i\in u$, yet the value $p_i(x_j^i)$ is different for the two numbers $i\in u$, then $f(u)=\langle 1, 0, j\rangle$ for some such $j$;
\item if the previous items fail and there are $b\in 2$ and $j\in k$ such that the point $x_j^{i_b}\notin\dom(p_{i_{1-b}})$ is $\Gamma$-connected with some point in $\dom(p_{i_{1-b}})\setminus\dom(p_{i_b})$, then let $f(u)=\langle 2, b, j\rangle$ for some such $b, j$;
\item if all of the previous items fail then $f(u)=$OK.
\end{itemize}

\noindent Use the Ramsey property of $m$ to find a set $a\subset m$ of cardinality $n$ which is homogeneous for $f$. As in the proof of Theorem~\ref{1theorem}, one can argue that the sets $\dom(p_i)$ for $i\in a$ form a $\Delta$-system, $\bigcup_{i\in a}p_i$ is a function, and every $\Gamma$-edge in $\dom(\bigcup_{i\in a}p_i)$ is a subset of $\dom(p_i)$ for some $i\in a$. That edge will not be monochromatic since $p_i$ is a $\Gamma$-coloring. In conclusion, $\bigcup_{i\in a}p_i$ is a common lower bound in $R(\Gamma)$ of the conditions $p_i$ for $i\in a$.
\end{proof}

\section{A balanced coloring poset}
\label{balancedsection}

Towards the proof of Theorem~\ref{maintheorem}, I must find a balanced forcing which generates the desired model of ZF+DC over the choiceless Solovay model. I use this opportunity to isolate a rather general scheme for coloring ternary hypergraphs of the following form:

\begin{definition}
\label{posetdefinition}
Let $\langle X, +\rangle$ be a standard Borel abelian group and $\Gamma$ a finite union of slim linear ternary hypergraphs on $X$. Let $C=\bigcup_n C_n$ be a countable set decomposed into infinitely many infinite sets.
The coloring poset $P$ consists of all $p$ such that there is a countable $\Gamma$-closed set $\dom(p)\subset X$ such that $p\colon\dom(p)\to C$ is a $\Gamma$-coloring. The ordering is defined by $q\leq p$ if $p\subseteq q$ and for every $\dom(p)$-orbit $a\subset\dom(q)$ there is a number $m\in\gw$ such that $q''a\subset\bigcup_{n\in m}C_m$.
\end{definition}

\noindent For the last condition in the definition, note that the group $\dom(p)$ acts on $X\setminus\dom(p)$ by addition, and $\dom(q)\setminus\dom(p)$ is a set invariant under this action. The choice of sets $C_n$ for $n\in\gw$ is clearly immaterial; for convenience, below I will set $C=\hefi\setminus\{0\}$ and $C_n=\{c\in C\colon |c|=n+1\}$.

\begin{proposition}
$P$ is an analytic, transitive, $\gs$-closed relation.
\end{proposition}

\begin{proof}
The analyticity of $P$ is clear. For the transitivity, assume that $r\leq q\leq p$ are conditions in $P$, and argue that $r\leq p$ must hold. For this, suppose that $a\subset\dom(r)\setminus\dom(p)$ is a $\dom(p)$-orbit; I must show that there is a number $m\in\gw$ such that $q''a\subset\bigcup_{n\in m}C_m$. There are two cases. Either, $a$ contains some element of $\dom(q)$. In this case, $a\subset\dom(q)$ holds, and the existence of the number $m$ follows from $q\leq p$. Or, $a\cap\dom(q)=0$. In this case, $a$ is a subset of a single $\dom(q)$-orbit and the existence of the number $m$ follows from $r\leq q$.

For the $\gs$-closure, it is clear that if $\langle p_n\colon n\in\gw\rangle$ is a descending sequence of conditions, then $\bigcup_np_n$ is its lower bound.
\end{proof}

\begin{proposition}
For every $p\in P$ and every $x\in X$ there is $q\leq p$ such that $x\in\dom(q)$.
\end{proposition}

\begin{proof}
First, choose $\dom(q)$ to be an arbitrary countable $\Gamma$-closed set containing $\dom(p)$ as a subset and $x$ as an element. Let $q\colon \dom(q)\to C$ be a function such that $q\restriction\dom(p)=p$ and $q\restriction\dom(q)\setminus\dom(p)$ is an injection whose range is a subset of $C_0$. Then $q$ is a $\Gamma$-coloring: any hyperedge $e\subset\dom(q)$ is either contained in $\dom(p)$ and not monochromatic as $p\in P$, or it contains more than one element of $\dom(q)\setminus\dom(p)$ by the $\Gamma$-closure of $\dom(p)$, and it is not monochromatic as $q\setminus p$ is an injection. Since $q\leq p$ holds, the proposition follows. 
\end{proof}

\begin{proposition}
\label{balanceproposition}
Let $\bar p\colon X\to C$ be a total $\Gamma$-coloring. Let $V[G_0]$ and $V[G_1]$ be generic extensions such that $V[G_0]\cap V[G_1]=V$. Let $p_0\in P\cap V[G_0]$ and $p_1\in P\cap V[G_1]$ be conditions such that $p_0\leq\bar p$ and $p_1\leq\bar p$. Then the conditions $p_0, p_1$ are compatible in $P$.
\end{proposition}

\noindent Note that the analytic poset $P$ is reinterpreted in every transitive model of set theory; in particular, the conditions $p_0, p_1$ are compatible in the poset $P$ as reinterpreted in $V[G_0, G_1]$ or any larger model.

\begin{proof}
Work in the model $V[G_0, G_1]$. 

\begin{claim}
$p_0\cup p_1$ is a function.
\end{claim}

\begin{proof}
This follows from the assumption that $V[G_0]\cap V[G_1]=V$: it then must be the case that $\dom(p_0)\cap\dom(p_1)\subset V$. The functions $p_0\restriction V$ and $p_1\restriction V$ are both the same, equal to $\bar p$. The claim follows.
\end{proof}

\noindent Of course, this does not end the proof; I have to find a countable $\Gamma$-closed set $\dom(q)\subset X$ which contains $\dom(p_0\cup p_1)$ and a condition $q\in P$ with this domain such that $q\leq p_0, p_1$ holds. To this end, for every $x\in X\setminus\dom(p_0\cup p_1)$ let $a_0(x)=\{x_0\in\dom(p_0)\setminus V\colon \exists x_1\in\dom(p_1)\setminus V\ \{x_0, x_1, x\}\in\Gamma\}$ and $a_1(x)=\{x_1\in\dom(p_1)\setminus V\colon \exists x_0\in\dom(p_0)\setminus V\ \{x_0, x_1, x\}\in\Gamma\}$.

\begin{claim}
Let $x\in X\setminus\dom(p_0\cup p_1)$ be a point. Each of the sets $a_0(x), a_1(x)$ is a subset of a finite union of $X\cap V$-orbits. In addition, 

\begin{enumerate}
\item if $b$ is a $\dom(p_0)$-orbit, then $\bigcup_{x\in b\setminus\dom(p_1)}a_1(x)$ is a subset of a finite union of $X\cap V$-orbits;
\item same with subscripts $0, 1$ interchanged.
\end{enumerate}
\end{claim}

\begin{proof}
Let $g_0, g_1, g_2$ be linear homomorphisms defining one of the slim linear components of the hypergraph $\Gamma$. For the first sentence of the claim, it will be enough to show that if $x_0, x'_0\in\dom(p_0)\setminus V$ and $x_1, x'_1\in \dom(p_1)\setminus V$ are points such that $g_0(x_0)+g_1(x_1)+g_2(x)=0$ and $g_0(x'_0)+g_1(x'_1)+g_2(x)=0$, then $x_0, x'_0$ are in the same $X\cap V$-orbit, and so are $x_1, x'_1$. To see this, subtract the latter equation from the former, and get $g_0(x_0-x'_0)=-g_1(x_1-x'_1)$. The left-hand side is in $V[G_0]$, the right-hand side is in $V[G_1]$, and since $V[G_0]\cap V[G_1]=V$, both are in $V$. Since $V$ is closed under inverses of $g_0$ and $g_1$, $x_0=x'_0\in X\cap V$ and $x_1-x'_1\in V$ as desired.

For (1) it will be enough to show that if $x_0, x'_0\in\dom(p_0)\setminus V$ and $x_1, x'_1\in \dom(p_1)\setminus V$ and $x, x'\in b$ are points such that $g_0(x_0)+g_1(x_1)+g_2(x)=0$ and $g_0(x'_0)+g_1(x'_1)+g_2(x')=0$, then $x_1, x'_1$ are in the same $X\cap V$-orbit. Again, subtract the latter equation from the former, getting $g_0(x_0-x'_0)-g_2(x_2-x'_2)=-g_1(x_1-x'_1)$. The left hand side is in $V[G_0]$ since $x_2-x'_2$ is; the right hand side is in $V[G_1]$. As before, this means that $-g_1(x_1-x'_1)\in V$ and $x_1-x'_1\in V$ as desired.
The proof of (2) is symmetric.
\end{proof}

\noindent Now, for each $x\in X\setminus\dom(p_0\cup p_1)$ let $n_x$ be the smallest number $n\in\gw$ such that there are no points $x_0\in\dom(p_0)\setminus V$ and $x_1\in\dom(p_1)\setminus V$ such that $p_0(x_0)\in C_n$, $p_1(x_1)\in C_n$, and $\{x_0, x_1, x\}\in\Gamma$. The claim (together with $p_0, p_1\leq\bar p$) shows that $n_x$ is well-defined and in addition, for every $\dom(p_0)$-orbit $b$, the numbers $n_x$ for $x\in b$ are bounded, and for every $\dom(p_1)$-orbit $c$, the numbers $n_x$ for $x\in c$ are bounded. Now, let $\dom(q)\subset X$ be any countable $\Gamma$-closed set containing $\dom(p\cup p_1)$, write $d=\dom(q)\setminus\dom(p_0\cup p_1)$, and let $q\colon\dom(q)\to C$ be any function such that $p_0\cup p_1\subset q$, for every $x\in d$ $q(x)\in C_{n_x}$, and $q\restriction d$ is an injection. Such a function clearly exists since the sets $C_n$ for every $n\in\gw$ are infinite. It will be enough to show that $q\in P$ and $q$ is a common lower bound of $p_0, p_1$.

First, I must argue that $q$ is a $\Gamma$-coloring. Let $e\subset\dom(q)$ be a $\Gamma$-hyperedge; I must argue that $e$ is not monochromatic. The discussion splits into several possible configurations.

\noindent\textbf{Case 1.} $e\subset\dom(p_0\cup p_1)$. In this case, a counting argument shows that there is $i\in 2$ such that $\dom(p_i)$ contains two points of $e$, and the $\Gamma$-closure of $\dom(p_i)$ shows that $\dom(p_i)$ contains all points of $e$. Then $e$ cannot be monochromatic since $p_i$ is a $\Gamma$-coloring.

\noindent\textbf{Case 2.} $e$ contains exactly one point, say $x$, in the set $d$. In this case, the closure properties of the sets $\dom(p_0)$ and $\dom(p_1)$ show that there must be a point $x_0\in\dom(p_0)\setminus V$ and a point $x_1\in\dom(p_1)\setminus V$ such that $e=\{x_0, x_1, x\}$. Then, $e$ is not monochromatic by the choice of the number $n_x$ and the fact that $q(x)\in C_{n_x}$.

\noindent\textbf{Case 3.} $e$ contains more than one point in the set $d$. In this case, $e$ is not monochromatic since $q\restriction d$ is an injection.

Finally, I must show that $q\leq p_0$; the proof of $q\leq p_1$ is symmetric. To this end, suppose that $b\subset\dom(q)\setminus\dom(p_0)$ is a $\dom(p_0)$-orbit; I must show that $q''b$ is a subset of the union of finitely many sets $C_n$ for $n\in\gw$. Write $b_0=b\cap\dom(p_1)$ and $b_1=b\setminus\dom(p_1)$. 

The first observation is that $b_0$ is a $\dom(\bar p)$-orbit: if $x, x'\in b_0$ then $x-x'\in V[G_1]$ as $x, x'\in V[G_1]$, and $x-x'\in V[G_0]$ as $b_0$ is a subset of $\dom(p_0)$-orbit. Since $V[G_0]\cap V[G_1]=V$, it follows that $x-x'\in V$; as $x, x'\in b_0$ were arbitrary, it follows that $b_0$ is a $\dom(\bar p)$-orbit. Since $p_1\leq\bar p$ is assumed, it follows that $q''b_0$ is a subset of the union of finitely many sets $C_n$ for $n\in\gw$.

The second observation is that the set $c=\bigcup_{x\in b_1}a_1(x)\subset\dom(p_1)$ is a subset of a finite union of $\bar p$-orbits, and since $p_1\leq\bar p$ is assumed, the set $q''c$ is covered by a finite union of the sets $C_n$ for $n\in\gw$. Let $m\in\gw$ be some number such that $C_m\cap q''c=0$. It follows from the definitions that $n_x\leq m$ holds for every $x\in b_1$, therefore $q''b_1\subset\bigcup_{n\leq m}C_n$ holds.

The two observations taken together show that $q''b$ is a subset of the union of finitely many sets $C_n$ for $n\in\gw$ as desired.
\end{proof}

\begin{corollary}
\label{balancecorollary}
The poset $P$ is balanced, and in fact placid in the sense of \cite[Definition 9.3.1]{z:geometric}.
\end{corollary}

\begin{proof}
It will be enough, for every condition $p\in P$, to find a total coloring $\bar p\colon X\to\hefi$ such that the range of $\bar p\setminus p$ consists of sets of cardinality at most two. For then, in any generic extension collapsing the cardinality of the ground model continuum to $\aleph_0$, it will be the case that $\bar p\in P$ and $\bar p\leq p$ holds. In addition, Proposition~\ref{balanceproposition} then shows that $\bar p$ will be exactly the balanced (or placid) virtual condition required in the definition of balance (or placidity).

Now, to obtain the extension of $p$ to $\bar p$, use Corollary~\ref{chromaticcorollary2} to find a $\Gamma$-coloring $c\colon X\to\hefi$ and use Corollary~\ref{chromaticcorollary1} to find a $\Delta(\dom(p))$-coloring $d\colon\mathbb{R}^2\setminus\dom(p)\to\hefi$. Define the function $\bar p$ by setting $\bar p(x)=p(x)$ if $x\in\dom(p)$, and $\bar p(x)=\langle c(x), d(x)\rangle$ if $x\notin\dom(p)$. This is literally an amalgamation of $p$ with $c$ and $d$ in the sense of Definition~\ref{amalgamationdefinition}, so Proposition~\ref{amalgamationproposition} shows that $\bar p$ is a $\Gamma$-coloring. By its definition, the range of $\bar p\setminus p$ consists of sets of cardinality at most two. This completes the proof.
\end{proof}

\section{The independence proofs}
\label{finalsection}

Finally, I am in the position to provide the proof of Theorem~\ref{maintheorem}. Let $\kappa$ be an inaccessible cardinal and let $W$ be the choiceless Solovay model derived from $\kappa$.  Let $G\subset P$ be a filter generic over $W$; I will show that $W[G]$ is a model witnessing Theorem~\ref{maintheorem}. There are several components of that proof.

\begin{proposition}
\label{vitaliproposition}
In the model $W[G]$, there is no Vitali set.
\end{proposition}

\begin{proof}
It will be enough to show that the poset $P$ has Ramsey control as in \cite[Section 11.6]{z:geometric}. This means that in ZFC, for a condition $p\in P$ there is a triple $\langle R, \leq, \mathcal{C}, \bar p\rangle$ such that

\begin{enumerate}
\item $R, \leq, \bar p$ are all definable from ordinal parameters and $p$;
\item $\mathcal{C}$ is a collection consisting of Ramsey centered subsets of $R$, each of which is definable from ordinal parameters and $p$, and $R=\bigcup\mathcal{C}$;
\\tem $\bar p$ is an $R$-name for a total $\Gamma$-coloring such that $R\Vdash\coll(\mathbb{R}, \aleph_0)\Vdash\bar p\leq \check p$.
\end{enumerate}

The poset $R$ is the finite support iteration of length $\gw_1$ in which the $\ga$-th iterand (in the model $\mathbb{R}^2\cap V[G_\gb\colon\gb\in\ga]$) is the poset $R_{\ga0}\times R_{\ga1}$ where $R_{\ga0}$ is the finite condition poset coloring the hypergraph $\Gamma$ from Theorem~\ref{1theorem}. To specify $R_{\ga1}$, use Example~\ref{criticalexample} to find a homomorphism $h_\ga\colon X_\ga\to Y_\ga$, where $X_{\ga}=X\setminus\bigcup_{\gb\in\ga}\{V[G_\gg\colon \gg\in\gb\}$ (if $\ga=0$) or $X\setminus\dom(p)$ if $\ga=0$ and $h_\ga$ is a homomorphism of the remainder graph $\Gamma_{X\setminus X_\ga}$ to some locally finite graph $\Delta_\ga$. Note that Example~\ref{criticalexample} provides $h_\ga, Y_\ga, \Delta_\ga$ which are definable from $p$ (if $\gamma=0$) or from the sequence of intermediate generic extensions (if $\gamma>0$). $R_{\ga1}$ will then be the finite condition poset adding a coloring of the graph $\Delta_\ga$.

 Each of these posets is suitably definably $\gs$-Ramsey-centered, and c.c.c. This means that the finite support iteration $R$ is c.c.c. and it is a union of Ramsey-centered pieces, each of which is definable from $p$ and some ordinal parameters. This is proved as in \cite[Proposition 11.6.3]{z:geometric}, noting that the Suslinity of the posets can be replaced with ordinal definability from the sequence of generic extensions.

At each stage of the iteration, $R_{\ga+1}$ adds a $\Gamma$-coloring $\dot c_\ga$ and $R_{\ga1}$ adds a coloring $\dot d_\ga$ of the remainder graph which is the pullback of the generic $\Delta_\ga$-coloring by the homomorphism $h_\ga$. In the end, let $\bar p$ be the name for the amalgamation of the coherent sequence $\langle p, \dot c_\ga, \dot d_\ga\colon\ga\in\gw_1\rangle$ of colorings as in Definition~\ref{amalgamationdefinition}. This will be a total $\Gamma$-coloring extending $p$ by Proposition~\ref{amalgamationproposition}; $R\Vdash\coll(\mathbb{R}, \aleph_0)\Vdash\bar p\leq \check p$ follows from the definitions.

The proof of the proposition is now completed by a reference to \cite[Theorem 11.6.5]{z:geometric} applied to the Vitali equivalence relation viewed as a Borel graph.
\end{proof}

\begin{proposition}
\label{e1proposition}
In the model $W[G]$, the equivalence relation $E_1$ does not have a complete countable section.
\end{proposition}

\begin{proof}
I will in fact show that the poset $P$ is nested balanced \cite[Definition 9.4.2]{z:geometric}. Then, by \cite[Theorem 9.4.4]{z:geometric}, in the $P$-extension of the Solovay model, the cardinality of the $\mathbb{E}_1$-quotient space is not smaller than the cardinality of the orbit space of any Polish group action. In particular, there is no complete $\mathbb{E}_1$ section $S\subset (2^\gw)^\gw$, since then the map sending any $\mathbb{E}_1$-class to its intersection with $S$ would be an injection of the $\mathbb{E}_1$-quotient space to the $\mathbb{F}_2$-quotient space, which is an orbits space of an action of the full permutation group $S_\infty$.

To recall what nested balance means, a sequence $\langle M_i\colon i\in\gw\rangle$ of transitive models of ZFC is \emph{coherent} \cite[Section 4.2]{z:geometric} if $M_0\supseteq M_1\supseteq\dots$ and for every ordinal $\ga$ and every number $i\in\gw$, the sequence $\langle M_j\cap V_\ga\colon j>i\rangle$ belongs to the model $M_i$. The sequence is \emph{choice-coherent} \cite[Section 4.3]{z:geometric} if for every ordinal $\ga$ there is a well-ordering $\prec$ of $V_\ga$ such that for every $i\in\gw$, $\prec\restriction M_i\in M_i$ holds. For a choice coherent sequence $\langle M_i\colon i\in\gw\rangle$, write $M_\gw$ for the intersection model $\bigcap_iM_i$. It is a transitive model of ZFC. 

To show that the poset $P$ is nested balanced, it is enough, for a given choice-coherent sequence $\langle M_i\colon i\in\gw\rangle$ of transitive models of ZFC, to build a descending sequence of suitably balanced conditions. In the context of the poset $P$, it means building a total coloring $\bar p\in M_0$ such that for each $i\in\gw+1$, the function $\bar p_i=\bar p_0\restriction M_i$ belongs to the model $M_i$, and for $i\leq j\leq\gw$, $\bar p_i\leq \bar p_j\leq p$ holds in any forcing extension which collapses $|M_0\cap\mathbb{R}|$ to $\aleph_0$.

To this end, let $\prec$ be a well-ordering of the set of all partial maps from $\mathbb{R}^2$ to $\hefi$ such that for each $i\leq\gw$, $\prec\restriction M_i\in M_i$ holds. 
Let $I=\gw+2$, and let $\leq$ be the reverse of the usual ordinal ordering on $I$. For each $i\in I$, define the real closed field $F_i$: if $i\leq\gw$ then $F_i=\mathbb{R}\cap M_i$, and $F_{\gw+1}=\supp(p)$. For each $i\in I$, define the colorings $c_i$ and $d_i$: if $i\in\gw+1$, then $c_i\colon F_i^2\to\hefi$ is the $\prec$-least $\Gamma$-coloring in the model $M_i$, and $d_i\colon F_i^2\setminus F_{i+1}^2\to \hefi$ is the $\prec$-least $\Delta_i$-coloring in the model $M_i$, where $\Delta_i$ is the remainder graph $\Delta(F_i^2, F_{i+1}^2)$. Also, $c_{\gw+1}=p$ and $d_{\gw+1}=0$. It is not difficult to see that $\langle I, \leq, F_i, c_i, d_i\colon i\in I\rangle$ is a coherent sequence of colorings as in Definition~\ref{coherentdefinition}. Let $\bar p$ be the amalgamation of these colorings as in Definition~\ref{amalgamationdefinition}. I claim that $\bar p$ is as required.

First of all, $\bar p$ is a $\Gamma$-coloring by Proposition~\ref{amalgamationproposition}. The definition of amalgamation, and the coherent choices of the colorings $c_i$ and $d_i$, immediately imply that for each $i\in\gw+1$, $\bar p_i=\bar p\restriction M_i$ belongs to the model $M_i$. Finally, $\bar p$ extends $p$, and by the definition of amalgamation, the range of $\bar p\setminus p$ consists of ordered pairs, i.e.\ of sets of cardinality at most two. It follows that for each $i\in j\in\gw+1$, $\bar p_i\leq \bar p_j\leq p$ holds in any forcing extension in which $|F_i|=\aleph_0$. The proof is complete.
\end{proof}

The proof of Theorem~\ref{maintheorem2} is quite different. It closely follows \cite[Section 11.8]{z:geometric}. It uses an auxiliary hypergraph and a simple proper forcing notion of independent interest.

\begin{definition}
Let $n\geq 2$ be a number. $\Delta_n$ is the $F_\gs$ $n$-ary hypergraph on the space $Y_n=n^\gw$, where a $\Delta_n$-hyperedge is a set $e\in [Y]^n$ such that there is a finite set $a\subset\gw$ and $z\in n^{\gw\setminus a}$ such that $e=\{y\in Y_n\colon  y\restriction \gw\setminus a=z$ and $y\restriction a$ is constant$\}$. 
\end{definition}

\noindent The hypergraph $\Delta_n$ is unfortunately not actionable in the sense of \cite[Definition 11.8.1]{z:geometric}, so the results of \cite[Section 11.8]{z:geometric} are not applicable verbatim and a restatement of the proof is necessary. There is a forcing notion associated with $\Delta_n$. Using the Hales--Jewett theorem repeatedly, produce a sequence of consecutive nonempty finite intervals $I_m$ of natural numbers for $m\in\gw$ such that for every partition of $n^{I_m}$ into $m$ pieces, one of the pieces contains a combinatorial line. Let $\phi_m$ be the submeasure on $n^{I_m}$ defined by $\phi_m(m)=$the smallest number of subsets of $n^{I_m}$, neither of which contains a combinatorial line, and together they cover $b$. Thus, $\phi_m(n^{I_m})\geq m$ the poset $R_n$ consists of sequences $r=\langle b_m\colon m\in\gw\rangle$ such that $b_m\subset n^{I^m}$ is a nonempty set and $\limsup_m\phi_m(b_m)=\infty$. The ordering on $R_n$ is that of coordinatewise inclusion. The generic point of $n^\gw$ added by the poset $R_m$ is the concatenation of the coordinatewise intersection of all conditions in the generic filter. The relevant properties of $R_n$ are verified as in \cite[Claim 11.8.4]{z:geometric}:

\begin{proposition}
The poset $R_n$ is proper, bounding, and adds no independent reals. In addition, for every condition $r\in R_n$, in the $R_n$-extension below $r$ there is a $\Delta_n$-hyperedge whose elements are all $R_n$-generic and meet the condition $r$.
\end{proposition}

\noindent The argument of \cite[Section 11.8]{z:geometric} can then be repeated to yield

\begin{proposition}
\label{deltanproposition}
In every compactly balanced extension of the choiceless Solovay model, the chromatic number of $\Delta_n$ is uncountable.
\end{proposition}

Finally, I return to abelian groups and linear equations.
Let $\langle X, +\rangle$ be a Polish abelian group, let $n\geq 3$ be a number, and let $g_i\colon X\to X$ for $i\in n$ be continuous homomorphisms such that $\gS_{i\in n}g_i=0$. Write $\Gamma$ for the hypergraph of arity $n$ containing a tuple $\langle x_i\colon n\rangle$ of pairwise distinct points if $\gS_{i\in n}g_i(x_i)=0$. Suppose that open neighborhood of $0$ in $X$ contains a $\Gamma$-hyperedge consisting of nonzero points.

Let $d$ be a left-invariant compatible metric on $X$. Since $X$ is abelian, $d$ is both-sided invariant, so complete. By recursion on $m\in\gw$, select nonzero pairwise distinct points $x_i(m)\in X$ for $i\in n$ so that

\begin{itemize}
\item $\gS_ig_i(x_i(m))=0$;
\item $d(0, x_i(m))<2^{-m}$ times the minimum $d$-distance between distinct points of the form $x_j(k)$ for $k\in m$ and $j\in n$.
\end{itemize}

\noindent In addition, define the map $\pi_n\colon Y_n\to X$ by setting $\pi_n(y)=\gS_m x_{y_m}(m)$. The completeness an invariance of the metric $d$ implies that the map $\pi_n$ is well-defined and continuous. In addition, the continuity of the homomorphisms $g_i$ for $i\in n$ show that if $\langle y_i\colon i\in n\rangle$ is a $\Delta_n$-hyperedge, then $\langle \pi_n(y_i)\colon i\in n\rangle$ is a $\Gamma$-hyperedge; the map $\pi_n$ is a homomorphism of $\Delta_n$ to $\Gamma$. In conjunction with Proposition~\ref{deltanproposition} this concludes the proof of Theorem~\ref{maintheorem2}.

\section{Problematic hypergraphs}
\label{problematicsection}

There are algebraic hypergraphs for which the existence of colorings yields remarkable consequences in ZF+DC. 

\begin{definition}
$\Gamma_R$ is the hypergraph of arity four on $\mathbb{R}^2$ consisting of all sets of the form $a\times b$ where both $a, b\subset\mathbb{R}^2$ are sets of cardinality two.
\end{definition}

\noindent To tie this hypergraph in with the previous sections, note that it is a subgraph of the linear hypergraph on $\mathbb{R}^2$ given by the equation $x_0-x_1+x_2-x_3=0$.
In ZFC, it is a well-known result (a conjunction of \cite{erdos:old} and \cite[Theorem 2]{komjath:three}) that existence of countable $\Gamma_R$-coloring is equivalent to the Continuum Hypothesis. In ZF, the situation is more nuanced.

\begin{theorem}
\label{rectangletheorem}
(ZF) Suppose that $\Gamma_R$ has countable chromatic number. Then

\begin{enumerate}
\item if there is an $\gw_1$-sequence of pairwise distinct reals, then there is such a sequence enumerating all reals;
\item there is a complete countable section for the $\mathbb{E}_1$ equivalence relation.
\end{enumerate}
\end{theorem}

\noindent The second item is the opposite of Theorem~\ref{maintheorem}(3). While I do not know if the existence of a Vitali set follows from the countable chromatic number of $\Gamma_R$, the first item shows that any negative consistency result in this direction must use an inaccessible cardinal. It also shows that the results of \cite{z:triangles, z:krull} require the inaccessible cardinal assumption as well. 

\begin{proof}
Let $c\colon \mathbb{R}^2\to\gw$ be a total $\Gamma_R$-coloring.

For (1), suppose that $s=\langle r_\ga\colon\ga\in\gw_1\rangle$ be a sequence of pairwise distinct reals. 

\begin{claim}
The model $HOD_{s,c}$ contains all reals.
\end{claim}

\begin{proof}
Suppose towards a contradiction that it does not. Let $r\in\mathbb{R}\setminus HOD_{s,c}$ be any real. Use a counting argument to find distinct countable ordinals $\ga, \gb$ such that the points $\langle r, r_\ga\rangle, \langle r, r_\gb\rangle\in\mathbb{R}^2$ receive the same color in $c$, say $n\in\gw$. The real $r$ cannot be unique with this property, since in such a case it would be definable and therefore an element of $HOD_{s, c}$. If $r'$ is a different real with this property, then te set $\{r, r'\}\times \{r_\ga, r_\gb\}$ is a monochromatic $\Gamma_R$-hyperedge with color $n$, a contradiction. 
\end{proof}

\noindent Now, work in the model $HOD_{s,c}$ and argue that the Continuum Hypothesis holds there as ZFC+countable chromatic number of $\Gamma_R$ implies CH; this will conclude the proof. This is in fact a known result \cite[Theorem 2]{komjath:three}; for completeness, let me include it. Work in ZFC, suppose that $d\colon\mathbb{R}^2\to\gw$ is a $\Gamma_R$-coloring, and assume towards a contradiction that the Continuum Hypothesis fails. Let $M$ be an elementary submodel of cardinality $\aleph_1$ of a large structure containing the coloring $d$. Let $r_0\in\mathbb{R}\setminus M$ be any point. By a counting argument, there must be a number $n\in\gw$ and distinct reals $s_0, s_1\in\mathbb{R}\cap M$ such that $d(s_0, r_0)=d(s_1, r_0)=n$. By elementarity of the model $M$, there has to be a point $r_1\in\mathbb{R}\cap M$ satisfying the latter equation with $r_0$ replaced with $r_1$. Then the set $\{s_0, s_1\}\times \{r_0, r_1\}$ is a $\Gamma_R$-hyperedge monochromatic in color $n$, contradicting the choice of the coloring $d$.

For (2), suppose that $x\in\mathbb{R}^\gw$ be a point, and for each $n\in\gw$ let $M_{nx}$ be the class of all sets hereditarily ordinally definable from $c$ and $x\restriction M_{nx}$. These models clearly form a sequence decreasing with respect to inclusion. Thus, the ordinals $(\gw_1)^{M_{nx}}$ form a non-increasing sequence, which has to stabilize at some number $m\in\gw$. The same proof in (1) shows that the sets $\mathbb{R}\cap M_{nx}$ stabilize at that same number $m\in\gw$. If this stabilized set is uncountable, then there is a bijection bewteen $\gw_1$ and $\mathbb{R}$ by (1), and the proof is complete. 

Suppose then that for every sequence $x\in\mathbb{R}^\gw$, the sets $\mathbb{R}\cap M_{nx}$ for $n\in\gw$ stabilize to some countable set. Then, even the sets $[x]_{\mathbb{E}_1}\cap M_{nx}$ stabilize in some countable set, call it $a(x)$. Note that the set $a(x)$ depends only on the $\mathbb{E}_1$-class of $x$. Thus, the set $\bigcup\{a(x)\colon x\in\mathbb{R}^\gw\}$ is a complete countable section for $\mathbb{E}_1$.
\end{proof}

\begin{definition}
$\Gamma_T$ is the hypergraph of arity three on $\mathbb{R}^2$ consisting of those sets for which adding a single point results in a $\Gamma_R$-hyperedge.
\end{definition}

\noindent Clearly, any $\Gamma_T$-coloring is a $\Gamma_R$-coloring; so, it is more difficult to find a $\Gamma_T$-coloring. Again, in ZFC, it is a well-known result ??? that existence of countable $\Gamma_T$-coloring is equivalent to the Continuum Hypothesis. Thus, in ZFC, hypergraphs $\Gamma_R$ and $\Gamma_T$ are conflated in this way, even though the construction of $\Gamma_T$-coloring is much more difficult.  In ZF, the distinction between the two hypergraphs becomes immediately apparent:

\begin{theorem}
(ZF) If $\Gamma_T$ has countable chromatic number, then there is a total countable-to-one map from $\mathbb{R}$ to $\gw_1$.
\end{theorem}

\noindent In particular, in ZF,  if $\Gamma_T$ has countable chromatic number, then every equivalence relation on a Polish space has countable complete section. I do not know if a Vitali set must exist. However, balanced forcing cannot be applied to obtain any consistency result in this direction, since countable-to-one maps from $\mathbb{R}$ to $\gw_1$ do not exist in balanced extensions of the choiceless Solovay model.

\begin{proof}
Let $c\colon \mathbb{R}^2\to\gw$ be a $\Gamma_T$ coloring. For $x\in\mathbb{R}$ write $M_x$ for the model of all sets hereditarily ordinally definable from $x$ and $c$. Note that $c\restriction M_x\in M_x$ holds. The argument divides into two cases.

\noindent\textbf{Case 1.} There is a real $x$ such that $\mathbb{R}\cap M_x$ is uncountable. In this case, apply Theorem~\ref{rectangletheorem}(1) to show that there is even a total injection from $\mathbb{R}$ to $\gw_1$.

\noindent\textbf{Case 2.} Case 1 fails. Let $\pi\colon\mathbb{R}\to\gw_1$ be the map defined by $\pi(x)=\gw_1^{M_x}$. The case assumption shows that the range of this map is indeed a subset of $\gw_1$. We will show that $\pi$ is in fact countable-to-one. Suppose towards contradiction that it is not, and let $\ga\in\gw_1$ be an ordinal such that the set $\{x\in\mathbb{R}\colon \pi(x)=\ga\}$ is uncountable. By the case assumption, there have to be points $x_0, x_1$ in this set such that $x_1\notin M_{x_0}$. We will reach a contradiction by a split into subcases.

\noindent\textbf{Case 1a.} Suppose first that $x_0\notin M_{x_1}$. Let $L_0$ be the line in $\mathbb{R}^2$ consisting of points whose $0$-th coordinate is equal to $x_0$ and let $L_1$ be the line in $\mathbb{R}^2$ consisting of points whose $1$-st coordinate is equal to $x_1$. Let $n=c(x_0, x_1)$. Then $\langle x_0, x_1\rangle$ is not the only point on $L_0$ which gets color $n$--otherwise $x_1$ would be definable from $x_0$. Let $\langle x_0, x_2\rangle\in L_0$ be a different point which gets color $n$. By the same argument, $\langle x_0, x_1\rangle$ is not the only point on $L_1$ which gets color $n$--otherwise $x_0$ would be definable from $x_1$. Let $\langle x_3, x_1\rangle\in L_1$ be a different point which gets color $n$. Then $\{\langle x_0, x_1\rangle, \langle x_0, x_2\rangle, \langle x_3, x_1\rangle\}$ is a monochromatic $\Gamma_T$-hyperedge of color $n$. A contradiction.

\noindent\textbf{Case 1b.} Assume now that $x_0\in M_{x_1}$. The set $\mathbb{R}\cap M_{x_0}$ then belongs to $M_{x_1}$ and must be uncountable there because the two models have the same $\gw_1$. By a counting argument in $M_{x_1}$, there must be distinct points $y_0, y_1\in\mathbb{R}\cap M_{x_0}$ such that $\langle y_0, x_1\rangle$ and $\langle y_1, x_1\rangle$ get the same $c$-color, say $n$. Now, $x_1$ cannot be the only point such that $\langle y_1, x_1\rangle$ gets the color $n$--otherwise $x_1$ would be definable from $y_1$ and then also from $x_0$. So, pick a point $z\in\mathbb{R}$ such that $c(y_1, z)=n$ and note that the set $\{\langle y_0, x_1\rangle, \langle y_1, x_1\rangle, \langle y_1, z\rangle\}$ is a $c$-monochromatic $\Gamma_T$-hyperedge of color $n$. This is a final contradiction.
\end{proof}

\bibliographystyle{plain} 
\bibliography{odkazy,zapletal,shelah}

\end{document}